\documentclass[10pt]{article}
\usepackage{amsmath, amsfonts, amsthm, amssymb, verbatim}
\usepackage{graphicx}
\usepackage[mathcal]{euscript}
\usepackage{amstext}
\usepackage{amssymb,mathrsfs}
\newcommand \be     {\begin{equation}}
\newcommand \ee     {\end{equation}}

\newcommand \del     \partial

\newcommand{\clg}[1]{{\mathcal{#1}}}

\newcommand{\td}[1]{{\tilde{#1}}}

\newcommand \AN      {\clg{AN}}

\def\XXint#1#2#3{{\setbox0=\hbox{$#1{#2#3}{\int}$}
\vcenter{\hbox{$#2#3$}}\kern-.5\wd0}}

 \newcommand{\C}{\mathbb C}

\newtheorem{theorem}{Theorem}[section]
\newtheorem{proposition}[theorem]{Proposition}
 \newtheorem{remark}[theorem]{Remark}
\newtheorem{lemma}[theorem]{Lemma}
\newtheorem{corollary}[theorem]{Corollary}
\newtheorem{definition}[theorem]{Definition}
\newtheorem{example}[theorem]{Example}

  \def\beqs{\begin{eqnarray*}}
 \def\enqs{\end{eqnarray*}}
 \def\beq{\begin{eqnarray}}
 \def\enq{\end{eqnarray}}

\begin{document}
\title{Operators That Attain their Minima}
\author{Xavier Carvajal$^1$, Wladimir Neves$^1$}

\date{}

\maketitle

\footnotetext[1]{ Instituto de Matem\'atica, Universidade Federal
do Rio de Janeiro, C.P. 68530, Cidade Universit\'aria 21945-970,
Rio de Janeiro, Brazil. E-mail: {\sl carvajal@im.ufrj.br,
wladimir@im.ufrj.br.}
\newline
\textit{To appear in: Bulletin of the Brazilian Mathematical  Society}
}
%
\begin{abstract}
In this paper we study the theory of operators on complex Hilbert
spaces, which attain their minima in the unit sphere. We prove
some important results concerning the characterization of the
$\clg{N}^*$, and also $\AN^*$ operators, see respectively Definition \ref{PN^*} and Definition \ref{PAN^*}. The injective 
property plays an important role in these operators, and shall be established by these classes. 
\end{abstract}
%

\maketitle

\tableofcontents

\section{Introduction} \label{IN}

 We shall be concentrated on this paper in a class of bounded
linear operators on complex Hilbert spaces, or on a subspace of
it, which attains their minima on the unit sphere. Hereupon  by a
subspace, we are always saying a closed subspace. Certainly, the
study of bounded linear operators that attain their minima have
some similarities with the ones that achieve their norm as studied
by the authors in \cite{XCWNIEOT}. Although, they not share the
same characteristics, for instance the injectivity property plays
an important role for that ones studied here, that is to say, the
class of operators that attains their minima.

We are going to study mostly the operators that satisfy the
$\clg{N}^*$ and $\clg{AN}^*$ properties, defined respectively in
Definition \ref{PN^*} and Definition \ref{PAN^*}. The class of
the $\clg{N}^*$ operators contains, for instance, the compact ones which are
non-injective (see Proposition \ref{ancompacto}). Then, to
introduce the theory, let $H$, $J$ be complex Hilbert spaces and
$\clg{L}(H,J)$ the Banach space of linear bounded operators from
$H$ to $J$. We emphasize the case that will appear most frequently
later, namely $\clg{L}(H,H)= \clg{L}(H)$. Furthermore, we recall
that, the space $\clg{L}(H,J)$ is a Banach space with the norm
\begin{equation}
\label{ON}
   \|T\|= \sup_{\|x\|_{H} \leq 1}\|T x\|_{J}= \sup_{\|x\|_{H} = 1}\|T x\|_{J}
\end{equation}
and, it is well known that, if $H$ has finite dimension, then the
closed unit ball in $H$ is compact (Heine-Borel Theorem) and the
above $supremum$ is a maximum. The important question whenever
such a $supremum$ is a maximum in the infinite dimensional case was
studied by the authors in \cite{XCWNIEOT}, where it is present
many characterizations for operators that achieve their norm.
Analogously, we now define the following value
\begin{equation}
\label{OM}
   [\,T\,]:= \inf_{\|x\|_{H} = 1}\|T x\|_{J}
\end{equation}
and ask when such an $infimum$ is a minimum. 
This is one of the
main issues of this article, which motivates the following
\begin{definition} \label{PN^*}
An operator $T \in \clg{L}(H,J)$ is called to satisfy the property
$\clg{N^*}$, when there exists an element $x_0$ in the unit
sphere, such that
$$
  [\,T\,]= \|T \, x_0\|_{J}.
$$
\end{definition}

We start the study by the following considerations:

{\bf 1.}\,  An operator with zero minimum on the unit sphere
should be non-injective in order to satisfy the property
$\clg{N}^*$. Indeed, if there exists an element $x_0$ in the unit
sphere, such that, $\|T \, x_0\|_{J}= [\,T\,]= 0$, it follows that
$Tx_0=0$, and for $T$ injective, $x_0=0$, which is a
contradiction. Equivalently, if $T$ is injective and satisfies the property
$\clg{N^*}$, then $[T]>0$.

\medskip
{\bf 2.}\, If $T$ is non-injective, then $T$ attains its minimum
and further $[T]=0$. In fact, when $T$ is non-injective, we have
$\textrm{Ker} \; T \neq \{0\}$, and hence there exists an element
$x \in \textrm{Ker} \; T$, $x \neq 0$, such that $Tx=0$.
Therefore, $\|T(x/\|x\|)\|=0=[T]$.

\medskip
{\bf 3.}\,
Let us consider $T \in \clg{L}(H,J)$ with $H$
finite dimensional. It is well-known that, $\textrm{dim}\, T(H) \le
\textrm{dim}\, H$, and since $S$ is a compact set, it follows that
$T(S)$ is compact. Therefore, applying the Weierstrass' Theorem, $T$
attains its minimum on $S$. We have the following cases:
\begin{itemize}
\item If $\textrm{dim}\, T(H) = \textrm{dim}\, H$, then
$\textrm{Ker}\, T=\{0\}$ and thus $T $ is injective. We conclude
in this case that $[T]>0$. 
\item If
$\textrm{dim}\, T(H) < \textrm{dim}\, H$, then $\textrm{Ker}\, T
\neq \{0\}$ and $T$ is non-injective. Thus $[T]=0$.
\end{itemize}
Due the above considerations, we have the following complete 
characterization of the non-injective operators:

$$
  \begin{tabular}{| c | c | c |}
  \hline
  \multicolumn{3}{|c|}{$T \in \clg{L}(H,J)$, non-injective}
  \\[8pt]
  \hline
     $\dim H< \infty$  &   \multicolumn{2}{c |} { $T \in \clg{N}^*$,}
  \\[8pt]
  \cline{1-1}
  $\dim H= \infty$     &   \multicolumn{2}{c |} { $[\,T\,]= 0$}                 
  \\[8pt]
 \hline
  \end{tabular}
$$
Moreover, we have the following partial 
characterization of the injective operators:

$$
  \begin{tabular}{| c | c | c |}
  \hline
  \multicolumn{3}{|c|}{$T \in \clg{L}(H,J)$, injective}
  \\[8pt]
  \hline
     $\dim H< \infty$  &   \multicolumn{2}{c |} { $T \in \clg{N}^*$, $[\,T\,] > 0$}
  \\[8pt]
  \hline
      &   \multicolumn{2}{c |} { $[\,T\,] = 0$, $T \notin \clg{N}^*$}    
  \\[8pt]
  \cline{2-3}
  $\dim H= \infty$      &       \multicolumn{2}{c |} { $[\,T\,] > 0$, $T \in \clg{N}^*$ $?$}
  \\[8pt]
 \hline
  \end{tabular}
$$

\bigskip
On the other hand, if the dimension of $H$ or the dimension of $J$ are finite and $T \in \clg{L}(H,J)$, then
there exists an $x$ in the closed unit ball in $H$ (indeed in the
boundary, i.e. the unit sphere), such that
$$
   [\,T\,]= \|T x\|_{J}.
$$
Therefore, any operator of finite range satisfies the property
$\clg{N}^*$. Moreover, an important class, which
we have the complete characterization of the
property $\clg{N}^*$, are the non-injective compact operators.
Indeed, we have the following
\begin{proposition}\label{ancompacto}
Let $T \in \clg{L}(H,J)$ be a compact operator, with $H$
infinite dimensional. Then, $T$  satisfies the property $\clg{N}^*$ if, and
only if, $T$ is non-injective.
\end{proposition}
\begin{proof}
1. First, let us show that, any compact operator $T \in
\clg{L}(H,J)$, with $H$ infinite dimensional has $[T]=0$. Indeed,
let $\{e_n\}$ be an infinite orthonormal set in $H$. Therefore,
applying Bessel's inequality, it follows for each $x \in H$
$$
  \sum_{n=1}^{\infty} |\langle x, e_n\rangle|^2 \le \|x\|_H^2.
$$
Thus for each $x\in H$, we have $\lim_{n \to \infty} \langle x,
e_n\rangle=0$. Consequently, the sequence $\{e_n\}$ converges
weakly to $0$ in $H$. Now, since $T$ is compact, $Te_n \to 0$ when
$n \to \infty$. Thus
$$
  0= \inf_{n}\|Te_n\|_J  \ge \inf_{\|e\|=1}\|Te\|_J= [T].
$$
2. Now, it follows from Consideration 2 in the preceding page that every non-injective operator 
on $H$ attains its minimum on the unit sphere. Conversely let $T$ be a compact operator on $H$
that attains its minimum on the unit sphere. Since $T$ is compact, from item 1 we have $[T]= 0$. 
Therefore, it follows from Consideration 1 in the preceding page that $T$ is non-injective.
\end{proof}

\bigskip
The restriction of a compact operator to a subspace is a compact
operator. 
Although, we have seen for instance that, injectiveness is an important property
w.r.t. the property $\clg{N}^*$. Since the restriction of a
non-injective operator is not necessarily non-injective, it does
not follow easy (even for the compact operator algebra) the
following property.

\begin{definition} \label{PAN^*}
We say that $T \in \clg{L}(H,J)$ is an $\AN^*$ operator, or to
satisfy the property $\clg{AN^*}$, when for all closed subspace $M
\subset H$ $(M \neq \{0\})$, $T|_M$ satisfies the property
$\clg{N^*}$.
\end{definition}
\begin{remark}\label{remark1}
Let $T \in L(H,J)$, if $\dim{H}< \infty$ or $\dim{J}<\infty$, then $T$ satisfy the property $\clg{AN^*}$.
\end{remark}
We stress that by a subspace, we always mean a closed subspace,
thus on the definition quoted above $M$ is always closed. Moreover, it is not difficult to see that,
one of the motivations to study the classes $\clg{N}^*$ and $\clg{AN}^*$ is related to show the injective property. 

\subsection{Notation and background} \label{NB}

At this point we fix the functional notation used in this paper,
and recall some well known results from functional analysis, we
address the references \cite{FHHSPZ}, \cite{WR}.

By $(H,\langle .,. \rangle)$ we always denote a complex Hilbert
space, $S$ will denote the unit sphere in $H$ and $B$ the closed
unit ball in $H$

The space $\clg{L}(H)$ is not only a Banach space, but also an
algebra. Moreover, we can define powers of $T \in \clg{L}(H)$,
that is $T^0=I$, where $I x= x$ for all $x \in H$, and generally
$T^n= T T^{n-1}$, $(n=1,2,\ldots)$. If $T \in \clg{L}(H,J)$, the
adjoint operator of $T$ is denoted by $T^* \in \clg{L}(J,H)$,
which satisfies $\|T^*\|=\|T\|$.

An operator $P \in \clg{L}(H)$ is called positive, when $\langle P
\, x,x \rangle \geq 0$, for all $x \in H$. Given an operator $T
\in \clg{L}(H,J)$, we denote by $P_T$, the unique operator called
the positive square root of $T^* T$, that is, $\langle P_T \, x,x
\rangle \geq 0$ for all $x \in H$ and $P_T^2= T^* T$. Moreover,
for $T \in \clg{L}(H)$ we recall the polar decomposition of $T$,
that is $T= U P_T$, where $U$ is a unitary operator $(U^*= U^{-1})$.
Not every $T \in \clg{L}(H)$ has a polar decomposition.
See \cite{WR}, Remark after Theorem 12.35. 

As usual, if $x,y \in H$, then $x \perp y$ means that $x$ is
orthogonal to $y$, i.e. $\langle x,y \rangle= 0$. Additionally, if
$M \subset H$, we define
$$
  M^\perp:= \{ x \in H : \langle x,y \rangle= 0, \text{for
  all} \; y \in M \},
$$
that is the orthogonal complement of $M$, which is a (closed)
subspace of $H$. If $M$ is a subspace of $H$, hence closed by
assumption, then we could write $H= M \oplus M^\perp$. 

Let $T \in \clg{L}(H)$. The numerical range of $T$ is defined as
$$
  W(T):= \{\langle Tx,x \rangle \in \C; \, x \in S \}.
$$
The Toeplitz-Hausdorff's Theorem asserts that $W(T)$ is a convex
set. Now, if $T \in \clg{L}(H)$ is a self-adjoint operator, then
$\|T\|= \sup_{x \in S} |\langle Tx,x \rangle|$. Therefore, for $P
\geq 0$, it follows that
\begin{equation}\label{normap}
  \|P\|= \sup_{x \in S} \langle
  Px,x \rangle=\sup W(P).
\end{equation}

Let $A \subset \C$ be a convex non-empty set. A number $\alpha \in
A$ is said to be an extreme point of $A$, when $\alpha= t \,
u+(1-t) \, v$, with $u,v \in A$ and $0<t<1$ implies, $\alpha=u=v$.
Extreme points could be defined in more abstract set. Moreover, we
recall the relation with convex sets given by the Krein-Milman
theorem, see \cite{WR}.

Finally, we recall some results and definitions in our paper
\cite{XCWNIEOT}.

\begin{definition} \label{PNAN}
An operator $T \in \clg{L}(H,J)$ is said to satisfy the property
$\clg{N}$, when there exists an element $x$ in the unit sphere,
such that
$$
  \|T\|= \|T \, x\|_{J}.
$$
Moreover, we say that $T$ is an $\AN$ operator, or to satisfy the
property $\clg{AN}$, when for all closed subspace $M \subset H$
$(M \neq \{0\})$, $T|_M$ satisfies the property $\clg{N}$.
\end{definition}

\begin{proposition}\label{l1}
Let $T \in \clg{L}(H)$ be a self-adjoint operator. Then, $T$
satisfies $\clg{N}$ if, and only if $\|T\|$ or $-\|T\|$ is an
eigenvalue of $T$.
\end{proposition}
It follows from the above proposition that, if $P \in \clg{L}(H)$
is a positive operator and there exists an element $x_0 \in S$,
such that $\|Px_0\|=\|P\|$, then
\begin{equation}
\label{PX0} P x_0= \|P\| x_0.
\end{equation}
Likewise, since $T$ satisfies $\clg{N}$ if, and only if $P_T$
satisfies $\clg{N}$. Indeed,
\begin{equation}\label{rq}
\|T\|=\| P_T \| \quad \textrm{and}\quad \forall x \in H, \quad
\|Tx\|=\| P_T x\|,
\end{equation}
hence we have the following
\begin{corollary}\label{p1}
An operator $T \in \clg{L}(H,J)$ satisfies $\clg{N}$ if, and only
if $\|T\|$ is an eigenvalue of $P_T$.
\end{corollary}
Now, we give the relation of the $\clg{N}$ condition and the
adjoint operator.
\begin{proposition}
\label{NCTAT}
 Let $T \in \clg{L}(H,J)$, then $T$ satisfies the condition $\clg{N}$ if,
 and only if the adjoint operator $T^*$ satisfies $\clg{N}$.
\end{proposition}
\begin{lemma}\label{l2}
Let $T \in \clg{L}(H)$ be an self-adjoint operator. Then, $T$
satisfies $\clg{N}$ if, and only if $\|T\|$ or $-\|T\|$ is an
extreme point of the numerical range $W(T)$.
\end{lemma}

\section{The $\clg{N}^*$ operators} \label{ANP}

As we said through the introduction, the main issue of this paper
is to study the operators that attain the minimum at the unit
sphere. We begin showing some important characteristics of the
$\clg{N}^*$ operators.

\begin{lemma}\label{l3}
If $T$ is self-adjoint operator on $H$, then for any $x \in H$ we
have
$$
\|T x\|^2 \ge [T]\left\langle Tx,x\right\rangle.
$$
\end{lemma}
\begin{proof}
Consider the operator $S:=T-[T]I$. Then, for each $x \in H$, we
have
$$
\|S x\|^2= \|T x\|^2-2[T]\left\langle T x,x\right\rangle+[T]^2\|
x\|^2 \ge 0.
$$
It follows that, $\|T x\|^2+[T]^2\| x\|^2 \ge 2[T]\left\langle T
x,x\right\rangle$. Now, since
$$
  \|T x\|^2 \ge [T]^2\|x\|^2,
$$
the proof follows.
\end{proof}

\begin{proposition}\label{PINS}
Let $P \in \clg{L}(H)$ be a positive operator. Then,
$$
  [P]= \inf \left\{\left\langle P x,x\right\rangle; x \in S\right\}.
$$
\end{proposition}
\begin{proof}
Define $m:=\inf \left\{\left\langle P x,x\right\rangle; x \in
S\right\}$. If the kernel of $P$ is different from zero, i.e.
${\rm Ker} P \neq \{0\}$, then $m=[P]=0$ and the result is trivial.
Now, suppose that ${\rm Ker}{P}=\{0\}$. We have
$$
m=\inf \left\{\left\langle P x,x\right\rangle; x \in S\right\} \le
[P],
$$
since $\langle Px,x\rangle \le \|P x\| \|x\|$. On the other hand,
using the positive square root of $P \ge 0$, it is known that, for
all $x,y \in H$,
\begin{equation}\label{positive}
\left|\left\langle P x,y\right\rangle\right|^2 \le \left\langle P
x,x\right\rangle \left\langle P y,y\right\rangle.
\end{equation}
Hence taking $y=P x$ in the above inequality, we have
\begin{equation}\label{eq2}
\|P x\|^2 \le \left\langle P x,x\right\rangle \left\langle
P\left(\dfrac{P x}{\|P x\|}\right), \dfrac{P x}{\|P
x\|}\right\rangle,
\end{equation}
and combining with Lemma \ref{l3},
we obtain
$$
\forall x \in H, \quad  [P] \le \left\langle P\left(\dfrac{P
x}{\|P x\|}\right), \dfrac{P x}{\|P x\|}\right\rangle.
$$
Therefore, for all $z \in P(H)$
$$
  [P] \,\,\|z\|^2 \le \left\langle P z,z\right\rangle.
$$
Consequently, as $\overline{P(H)}=H$, we conclude     that
$$
  \forall z \in H, \quad [P] \,\,\|z\|^2 \le \left\langle P
z,z\right\rangle,
$$
and it proves the proposition.
\end{proof}

\bigskip
Given $T \in \clg{L}(H)$, it is well known that,
$\left\|T\right\|=\sup_{x,y \in S}|\left\langle T x ,y
\right\rangle|$, see \cite{WR}. Therefore, we could conjecture
that
$$
  [T]=\inf_{x,y \in S} \, |\left\langle T x ,y\right\rangle|.
$$
In fact, this is false. Let us consider the following

\begin{example}
\label{EXAMPLE10}
Let $T: l^2 \to l^2$, $(x_j) \mapsto (\lambda_j x_j)$, with
$$
  \lambda_1>\lambda_2> \cdots >\lambda>0, \quad \lambda_j \searrow
\lambda.
$$
Then, $T \ge 0$ and it is easy to see that $T$
does not satisfy $\clg{N^*}$. Indeed, we have $[T]= \lambda$, since
$$
     \|T x \|^2= \sum_{j=1}^\infty \lambda_j^2 \, x_j^2 > \lambda^2 \, \| x \|^2, 
     \quad \forall x \neq 0, \| T \; \frac{x}{\|x\|} \| > \lambda
$$
and, if $(e_j)$ is the orthonormal canonical base of $l^2$, then
$\| T e_j \|= \lambda_j \to \lambda$. On the other hand,
$\inf_{x,y \in S}|\left\langle T x ,y \right\rangle|= 0$.
\end{example}

\bigskip
\begin{remark} If $P \geq 0$ and $[P]=\|P\|$, then $P=[P] I$. In
fact by \eqref{normap} and Proposition \ref{PINS} we have for any
$x \in H$ that
$$
[P]\,\|x\|^{2} \le \langle P x ,x \rangle \le
\|P\|\,\|x\|^{2}=[P]\,\|x\|^{2},
$$
and these inequalities give
$$
\langle (P -[P]I)x ,x \rangle =0.
$$
Therefore, since $P -[P]I \ge 0$, we conclude that $Px =[P]x$.
\end{remark}

\medskip
One observes that, if $P \geq 0$, then $P^n \ge 0$, $(n=1,2,
\dots)$. This result is easily obtained by induction. Now, 
if $P \ge 0$, then it is not difficult to show using
\eqref{positive} that, for each $n \geq 1$
\begin{equation}\label{Pn}
\|P^n\|=\|P\|^n.
\end{equation}

The following proposition shows that $[\, _\cdot \,]$ also has the
property \eqref{Pn}.
\begin{proposition}
\label{PPN} Let $P \in \clg{L}(H)$ be a positive operator. Then,
$$
  [P^n]=[P]^n.
$$
\end{proposition}

\begin{proof}
By Proposition \ref{PINS}, we have
\begin{align}
[P^2]&= \inf_{ x\in S}\left\langle P^2 x,x \right\rangle \nonumber
      = \inf_{ x\in S}\left\langle P x, P x \right\rangle \nonumber\\
     & = [P]^2.
\end{align}
Moreover, if $x \in S$, then
\begin{align}
\left\langle P^{n+1} x,x \right\rangle = &\left\langle P^{n-1} (P
x),P x \right\rangle  \nonumber
\\
\geq &[P^{n-1}]\,\,\left\|P x\right\|^2\nonumber\\
\geq &[P^{n-1}]\,\,[P]^2.
\end{align}
Consequently, we have that $ [P^{n+1}] \ge [P^{n-1}]\,\,[P]^2$
and, it follows by induction that
$$
  [P^{n}] \ge [P]^n, \quad (n=1,2, \dots).
$$
On the other hand, for any $x\in S$, $(n=1,2 \dots)$, by Lemma
\ref{l3}, we get
\begin{align}
[P^{n-1}]\left\langle P^{n+1} x,x \right\rangle =&
[P^{n-1}]\left\langle
P^{n-1} (P x),P x \right\rangle  \nonumber\\
\le & \left\| P^{n-1}(Px)\right\|^2,
\end{align}
thus for any $n \in \mathbf{N}$, we obtain
$$
  [P^{n-1}]\,\,[P^{n+1}]\le [P^n]^2.
$$
By induction, we get
$$
[P^{n}]\le [P]^n, \quad (n=1,2 \dots).
$$
Indeed, assuming that  $[P^{n}]\le [P]^n$, it follows that
$$
     [P^{n-1}] \, [P^{n+1}] \leq [P^{n}]^2 \leq  [P]^{2n}= [P]^{n-1} \, [P]^{n+1} \leq [P^{n-1}] \, [P]^{n+1},
$$
and therefore $[P^{n+1}] \le [P]^{n+1}$, which proves the proposition.
\end{proof}

\bigskip
Similarly to the property $\clg{N}$ see Proposition \ref{l1} and
Lemma \ref{l2}, we have the following
\begin{proposition}\label{propxav1}
Let $P \in \clg{L}(H)$ be a positive operator.

i) $P$ satisfies $\clg{N^*}$ if, and only if $[P]$ is an
eigenvalue of $P$.

\medskip
ii) $P$ satisfies $\clg{N^*}$ if, and only if $[P]$ is an extreme
point of $W(P)$.
\end{proposition}

\begin{proof} 1. In order to prove the item $(i)$, first we suppose that
$P$ satisfies $\clg{N^*}$, i.e., there exists $x_0 \in S$, such
that $\left\|P x_0\right\|= [P]$. Now, if $[P]=0$, then it is
obvious that $0$ it is an eigenvalue. Therefore, we assume that
$[P]>0$ and, we have that
$$
\left\langle (P^2-[P]^2 I)x_0,x_0\right\rangle=\|P
x_0\|^2-[P]^2=0.
$$
Since $P^2-[P]^2 I \ge 0$ and taking $z= P x_0- [P] x_0$, it follows
that
$$
  P z+ [P]z=0.
$$
Thus $\left\langle P z,z\right\rangle=-[P]\,\, \|z\|^2$ and as
$P\ge 0$ we concludes that $z=0$, hence $[P]$ is an eigenvalue.
Now, it is obvious that if $[P]$ is an eigenvalue of $P$, then $P$
satisfies $\clg{N^*}$.

2. The proof of item $(ii)$ is similar with that one given at
Proposition \ref{l1}.
\end{proof}

\bigskip
The next example is an injective operator, which does not satisfy 
the $\clg{N^*}$ property.

\begin{example}\label{ex4}
Consider the operator of Example \ref{EXAMPLE10}, that is 
$$
T: l^{2} \to l^{2}; \quad x \mapsto \left(\lambda_1 x_1, \lambda_2 x_2,
\lambda_3 x_3, \dots\right), 
$$
with 
$$
  \lambda_j \searrow \lambda > 0, \quad 
 \lambda_1 > \lambda_2> \dots .
$$
It is not difficult to verify that, $T \ge 0$ is an injective operator.
Moreover, we have that $T$ does not satisfy $\clg{N^*}$ property, which follows also  
since the numerical range of $T$ is the interval $(\lambda, \lambda_1 ]$ and
$[T]= \lambda$ is not an extreme point of the
numerical range.
\end{example}

\medskip
Given an operator $T$ on $H$ which satisfies the $\clg{N}$
condition, it is not necessarily true that $T^2$ also satisfies
$\clg{N}$. In fact, the following example shows an operator that
satisfies $\clg{N}$ and, such that $T^{2}$ does not satisfy
$\clg{N}$.

\begin{example}
\label{ET2} Let $T:l^2 \rightarrow l^2$, $(x_1,x_2,x_3,\ldots)
\mapsto (\lambda x_2,0,\lambda_1 x_3,\lambda_2 x_4, \dots )$, with
$$
  0<\lambda_1 < \lambda_2< \dots< \lambda, \quad \lambda_j \nearrow
  \lambda.
$$
Then, $T$ satisfies $\clg{N}$ condition, since
$$
  \lambda=\|T\|=\|T e_2\|,
$$
where $e_2=(0,1,0, \dots)$. But, we could show easily that
$$
 T^{2}(x_1,x_2,x_3,\ldots)= (0,0,\lambda_1^{2} x_3,\lambda_2^{2} x_4, \dots ),
$$
does not satisfy the property $\clg{N}$.
\end{example}
\begin{example}
\label{ET233} Let $T:l^2 \rightarrow l^2$, $(x_1,x_2,x_3,\ldots)
\mapsto (\lambda x_2,0,\lambda_1 x_3,\lambda_2 x_4, \dots )$, with
$$
  \lambda_1 > \lambda_2> \dots> \lambda > 0, \quad \lambda_j \searrow
  \lambda.
$$
Then, $T$ satisfies $\clg{N}^*$ condition, since
$$
  \lambda=[T]=\|T e_2\|,
$$
where $e_2=(0,1,0, \dots)$. But, similarly as above we could show that
$$
 T^{2}(x_1,x_2,x_3,\ldots)= (0,0,\lambda_1^{2} x_3,\lambda_2^{2} x_4, \dots ),
$$
does not satisfy the property $\clg{N}^*$.
\end{example}

\bigskip
\begin{proposition}\label{pn}
Let $P \in \clg{L}(H)$, $P \geq 0$ and $n$ be a positive integer.

i) $P$ satisfies $\clg{N}$ if, and only if $P^n$ satisfies
$\clg{N}$.

\medskip
ii) $P$ satisfies $\clg{N^*}$ if, and only if $P^n$ satisfies
$\clg{N^*}$.
\end{proposition}

\begin{proof}
1. First, we show $(i)$. When $P= 0$, the result is trivial, hence
we assume $P>0$. If $P^n$ satisfies $\clg{N}$, then by Corollary
\ref{p1}, we have
$$
  P^n x_0= \|P^n\|x_0= \|P\|^n x_0,
$$
for some $x_0 \in S$. Therefore, an algebraic manipulation gives
$$
  P\left(\dfrac{P^{n-1} x_0}{\|P\|^{n-1}}\right)=\|P\| x_0.
$$
Now, since $\|\, P^{n-1} x_0/\|P\|^{n-1} \,\| \le 1$, we obtain
that $P$ satisfies $\clg{N}$. In order to show that, if $P$
satisfies $\clg{N}$ condition, then $P^n$ satisfies $\clg{N}$, the
proof follows easily applying Corollary \ref{p1}.

2. The proof of the item $(ii)$ is similar.
\end{proof}

\bigskip
\begin{proposition}\label{p3} Let $P \in \clg{L}(H)$, $P \geq 0$, $n$ and
$k$ be positive integers. Define
$$
  T_n:= \|P\|^{n} I-P^{n}, \quad\tilde{ T}_n:= P^{n}-[P]^{n}I.
$$

i) $P$ satisfies $\clg{N^*}$ $<=>$ $T_n$ satisfies $\clg{N}$ $<=>$ $(T_n)^k$
satisfies $\clg{N}$.

\medskip
ii) $P$ satisfies $\clg{N}$ $<=>$ $\tilde{T}_n^*$ satisfies $\clg{N^*}$ $<=>$ $(\tilde{T}_n^*)^k$
satisfies $\clg{N^*}$.
\end{proposition}

\begin{proof}
1. Let us show $(i)$. If $P$ satisfies $\clg{N^*}$, then by Proposition \ref{propxav1}  there
exists $x_0 \in S$, such that $P x_0=[P] x_0$, it follows that $P^n x_0=[P]^n x_0$, and since for each $x \in H$,
$$
  \langle T_n x,x \rangle= \|P\|^{n}\|x\|^{2} - \langle P^{n}x,x
  \rangle\geq 0,
$$
we have
%
$$
  \|T_n\|= \|P\|^{n}-[P]^{n}= \|P\|^{n}\|x_0\|^{2}-\langle
   P^{n}x_0,x_0\rangle=\langle T_n x_0, x_0\rangle.
$$
Consequently, we obtain that $T_n= \|P\|^{n}I-P^{n}$ satisfies
$\clg{N}$. Now, if $T_n \ge 0$ satisfies $\clg{N}$, then there
exist $x_0 \in S$, such that
$$
  \langle T_n x_0,x_0 \rangle= \|P\|^{n} - \langle
  P^{n} x_0,x_0 \rangle= \|T_n\|= \|P\|^{n}-[P]^{n}
$$
and this implies that, $\langle
P^{n}x_0,x_0\rangle=[P]^{n}=[P^{n}]$. By Proposition \ref{pn} we
conclude that, $P$ satisfies $\clg{N^*}$. Finally, since $T_n \geq
0$, by Proposition \ref{pn} $T_n$ satisfies $\clg{N}$ if, and only
if $(T_n)^k$ satisfies $\clg{N}$, which completes the proof of
$(i)$.

2. The proof of the item $(ii)$ is similar.
\end{proof}

\bigskip
\begin{proposition}
Let $P \in \clg{L}(H)$, $P \ge 0$ and $(p_n)$ be a sequence of
polynomials with positive coefficients, such that
\begin{equation}\label{converg1}
p_n(P) \to S \quad \textrm{in} \,\,\clg{L}(H).
\end{equation}
If $P$ satisfies $\clg{N}$, then $S$ satisfies $\clg{N}$.
\end{proposition}

\begin{proof}
If $\left\|P x_0\right\|=\left\|P\right\|$, for some $x_0 \in S$,
then we have from Corollary \ref{p1} and Proposition \ref{pn} that
\begin{align}\label{pn2}
p_n(P)x_0=& \alpha_0 x_0+\alpha_1 P(x_0)+\cdots+\alpha_n P^n(x_0) \nonumber\\
=& p_n(\left\|P\right\|)x_0.
\end{align}
The equality \eqref{pn2} gives
$\left\|p_n(P)\right\|=p_n(\left\|P\right\|)$ and by
\eqref{converg1}, we obtain
$$
p_n(\|P\|) \to \left\|S\right\|.
$$
Now, the convergence \eqref{converg1} and the equality \eqref{pn2}
also imply that
$$
p_n(\|P\|) \to \left\|S x_0\right\|.
$$
Therefore $\left\|S\right\|=\left\|S x_0\right\|$.
\end{proof}

\begin{example}Let $P \in \clg{L}(H)$ be a positive operator. If $P$ satisfies $\clg{N}$,
then the exponential operator $\exp(P)$ satisfies $\clg{N}$.
Moreover, we have
$$
\|p_n(P)\|=\sum_{j=1}^{n}\dfrac{\|P\|^{j}}{j!} \to e^{\|P\|},
$$
as $n \to \infty$. Consequently, we have $\|\exp(P)\|$=
$e^{\|P\|}$.
\end{example}

\bigskip
\section{The $\clg{AN}^*$ operators} \label{ANP}

In this section, we are going to study the operators that
satisfy Definition \ref{PAN^*}, that is, the $\clg{AN^*}$
operators.

\medskip
As already seen in \cite{XCWNIEOT}, any compact operator $T$ in
$\clg{L}(H,J)$ is an $\AN$ operator. Indeed, if $M$ is any closed
subspace of $H$, then $T|_M$ is compact and therefore satisfies
$\clg{N}$. Consequently, the algebra of compact operators carries
$\clg{AN}$ out. Although, for the $\clg{AN^*}$ condition as we have 
showed at Example \ref{ex4}, with $\lambda= 0$, a compact operator
$T$ does not necessarily satisfy the $\clg{AN^*}$ property. 
Note that, if $T \in \clg{L}(H,J)$ is a compact operator with $[T]>0$, 
which implies that  $\dim{H} <\infty$ necessarily (see proof of 
Proposition \ref{ancompacto}), 
then $T$ satisfies $\clg{AN^*}$.

\bigskip
Now, since an orthogonal projection is a partial isometry, it follows
that any projection
satisfies the properties $\clg{N}$ and $\clg{N^*}$. Although, it
was showed in \cite{XCWNIEOT} that there exist a projection, which
does not satisfy the $\clg{AN}$ property. Similarly, it is not
necessarily true that each projection satisfies the $\clg{AN^*}$
property. In fact, we have the following

\begin{example}
\label{ENAN} Let $X$ be the subspace of $l^2$
of all $x$ of the form
$$
  x= (x_1, x_2, x_2, x_3, x_4, x_4, x_5, \ldots)
$$
and $P$ is the projection on $X$, i.e., $P:l^2 \to l^2$,
$$
  P(x_1,x_2,x_3,\ldots)= \Big(x_1, \frac{x_2+x_3}{2},
  \frac{x_2+x_3}{2}, x_4, \frac{x_5+x_6}{2}, \frac{x_5+x_6}{2},
  x_7, \ldots \Big).
$$
Now, let $M$ be a subspace of $l^2$, defined as
$$
  M:= \{x \in l^2 : x= (x_1,x_1,x_2,x_2,x_2,x_3,x_3,x_3,x_4,x_4,x_4,\ldots)
  \}.
$$
It follows that, $M \cap X= \{0\}$. Set $P|_M \equiv T: M \to
l^2$, hence
$$
  T(x_1,x_1,x_2,x_2,x_2,\ldots)= \Big(x_1, \frac{x_1+x_2}{2},
  \frac{x_1+x_2}{2}, x_2, \frac{x_2+x_3}{2}, \frac{x_2+x_3}{2},
  x_3, \ldots \Big).
$$
For each $x \in M \cap S$, we compute the norm of $T x$. First, we
have
\be \label{NX}
  1= \|x\|^2 = 2 x_1^2 + 3 \, \sum_{j=2}^\infty x_j^2,
\ee
hence it follows that
\be \label{TNX}
  \begin{aligned}
  \|Tx\|^2&= \sum_{j=1}^\infty x_j^2 + 2 \, \sum_{j=1}^\infty \big(\frac{x_j+x_{j+1}}{2}\big)^2
\\
  &= x_1^2 + \sum_{j=2}^\infty x_j^2 + \frac{x_1^2}{2}  + \sum_{j=2}^\infty x_j^2 + \sum_{j=1}^\infty
  x_j \, x_{j+1}
  \\
  &= \frac{2}{3} + \frac{x_1^2}{6} + \sum_{j=1}^\infty
  x_j \, x_{j+1},
  \end{aligned}
\ee
where we have used \eqref{NX}.
\medskip
We take a convenient sequence $\{t^n\}$ contained in
$M \cap S$, to show that $T$ does not satisfy the $\clg{AN^*}$
condition. Indeed,
we will show that, for all $x \in M \cap S$,
$$
  \|Tx\| > [T]= \frac{1}{\sqrt3}.
$$
We consider the sequence $\{t^n\}_{n=1}^{\infty}\subset M \cap S$,
$$
   t^n= \{t_1^n, t_1^n, t_2^n, t_2^n, t_2^n, \ldots  t_n^n, t_n^n, t_n^n, \ldots\}, \quad \|t^n\|= 1,
$$
$t^n$ defined by
\begin{equation*} 
t^n_j= \left \{
\begin{aligned}
      \frac{(-1)^j}{\sqrt{3 (n-1) + 2}} \quad j&= 1,\ldots,n,
      \\
      0 \quad \quad \quad \quad j&>n.
\end{aligned}
\right.
\end{equation*}
It follows that
$$
    \begin{aligned}
  \|T t^n\|^2&= \frac{2}{3} + \frac{(t_1^n)^2}{6} + \sum_{j=1}^{n-1} t_j^n \, t_{j+1}^n
\\  
  &= \frac{2}{3} + \frac{1}{6(3(n-1)+2)} - \sum_{j=1}^{n-1} \frac{1}{3(n-1)+2}
\\
  &= \frac{2}{3} + \frac{7-6n}{6(3(n-1)+2)}.
  \end{aligned}
$$
Then, we obtain
$$
  \lim_{n \to \infty} \|T t^n\|^2= \frac{2}{3} - \frac{6}{18}= \frac{1}{3}.
$$
Now using \eqref{NX}, we have for any $x \in M \cap S$

\begin{align}\label{projn*}
  \|Tx\|^2&= \sum_{j=1}^\infty x_j^2 + 2 \, \sum_{j=1}^\infty \big(\frac{x_j+x_{j+1}}{2}\big)^2
\nonumber\\
  &= x_1^2 + \frac{1-2x_1^2}{3}+\, \sum_{j=1}^\infty \frac{\big(x_j+x_{j+1}\big)^2}{2}\nonumber\\
  &= \frac13+ \frac{x_1^2}{3}+\, \sum_{j=1}^\infty
  \frac{\big(x_j+x_{j+1}\big)^2}{2}.
  \end{align}
Consequently, for any $x \in M \cap S$ we have $\|Tx\|^{2} \ge
1/3$ and hence $[T]=1/\sqrt3$. Therefore, we find out that $T$ does
not satisfy $\clg{N^*}$. Indeed, if there exists an element
$\tilde{x}$  in $S \cap M$, such that $\|T\tilde{x}\|=
[T]=1/\sqrt3$, then by \eqref{projn*}
$$
  \frac{\tilde{x}_1^2}{3}+\, \sum_{j=1}^\infty \frac{\big(\tilde{x}_j+\tilde{x}_{j+1}\big)^2}{2}=0,
$$
and this equation implies that $\tilde{x}_1=0$. Moreover,
$\tilde{x}_j+\tilde{x}_{j+1}=0$ for $j=1,2, \dots$, it follows
that $\tilde{x}_j=0$ for $j=1,2, \dots$, which is a contradiction
since that $\|\tilde{x} \|=1$. Hence, $T$ does not satisfy
$\clg{N^*}$.
\end{example}

Therefore, we have proved the following result.

\begin{lemma}
\label{CE1} Let $P$ be an orthogonal projection. Then $P$ does not
necessarily satisfy $\clg{AN^*}$ property.
\end{lemma}

\bigskip
The next proposition will be used as a proof of the next
theorem, but it is important by itself.

\begin{proposition}
\label{PCOM} Let $R$ be an isometry on $H$ and $T \in \clg{L}(H)$
an $\clg{AN}^*$ operator. Then, $TR$ and $RT$ satisfy the property
$\clg{AN}^*$.
\end{proposition}

\begin{proof}
Similar to that one given at Proposition 3.2 in \cite{XCWNIEOT}.
\end{proof}

\bigskip
Subsequently, we recall a well known definition for equivalent operators.

\begin{definition}
\label{EQO} The operators $T \in \clg{L}(H)$ and $S \in
\clg{L}(J)$ are called unitarily equivalents, when there exists a
unitary operator $U$ on $\clg{L}(J,H)$, such that
$$
  U^* \, T \, U= S.
$$
\end{definition}
In fact, if $T$ and $S$ are unitarily equivalents, then there is no
criterion based only on the geometry of the Hilbert space, in such
a way that, $T$ could be distinguished from $S$. Therefore, since
$T$ and $S$ are abstractly the same operator, it is natural to
conjecture that some characteristic endowed by $T$ must be
satisfied by $S$, and vice versa.

\begin{theorem}
\label{UEO} Let $T,S$ be two unitarily equivalent operators. Then,
$T$ is $\clg{AN}^*$ operator if, and only if $S$ is an $\clg{AN}^*$
operator.
\end{theorem}

\begin{proof}
Assume that $U$ is a unitary operator such that $U^*\, T \,
U= S$, hence $TU= US$. Since $U$ is an isometry, by Proposition
\ref{PCOM} if $T$ satisfies $\clg{AN}^*$, then $TU$ satisfies
$\clg{AN}^*$. Moreover, it follows that, $US$ also satisfies
$\clg{AN}^*$. Once more, conforming to Proposition \ref{PCOM}, we
have that $S$ satisfies property $\clg{AN}^*$.
\end{proof}

\begin{remark} \label{TPT}
Given $T \in \clg{L}(H,J)$, we recall that $P_T$ was defined as
the positive square root of $T^*T$. Therefore, $T$ satisfies
$\clg{AN}^*$ if, and only if $P_T$ satisfies $\clg{AN}^*$, see
\eqref{rq}. Consequently, it is enough to establish the condition
$\clg{AN}^*$ for positive operators.
\end{remark}

\begin{proposition}
An operator $T \in \clg{L}(H,J)$ satisfies the property $\clg{AN}^*$
if, and only if, for all orthogonal projection $Q \in \clg{L}(H)$,
the composition $TQ$ satisfies $\clg{N}^*$.
\end{proposition}

\begin{proof}
Let $M$ be a closed subspace of $H$ and $Q$ an orthogonal
projection on $M$. Then, we have
$$
  [TQ]= [T|_M].
$$
\end{proof}

\begin{lemma}
\label{IFR} Let $R \in \clg{L}(H)$ be an operator of finite rank.
Then $I + R$ is an $\clg{AN}^*$ operator.
\end{lemma}

\begin{proof}
We suppose that $\dim R(H)=n$. Hence we have
$$
  R x = \sum_{j=1}^n \lambda_j \langle x, e_j\rangle \, e_j
$$
where $\{e_j\}_{j=1}^n$ is an orthonormal set of $H$ and
$\lambda_j \ge 0$, $(j=1,2, \dots n)$. Let $M_n$ be the subspace
generated by $\{ e_1,\ldots,e_n \}$, thus we could write
$$
  H= M_n \oplus M_n^\perp.
$$
Moreover, for any $x \in H$, $x= x_1 + x_2$, such that
$$
  x_1= \sum_{j=1}^n \langle x,e_j \rangle \, e_j
  \quad \text{and} \quad
  x_2= \sum_{\alpha \in A} \langle x, \td{e}_\alpha \rangle \, \td{e}_\alpha,
$$
where $\{\td{e}_\alpha\}_{\alpha \in A}$ is an orthonormal basis of
$M_n^\perp$, $\langle \td{e}_\alpha,e_j \rangle= 0$, for all
$j=1,\ldots,n$, $\alpha \in A$ and $\langle \td{e}_\alpha, \td{e}_\beta
\rangle= \delta_{\alpha \beta}$ for each $\alpha, \beta \in A$.
Now, define $T:= I + R$, then for each $x \in H$,
$$
  Tx= \sum_{j=1}^n \langle x,e_j \rangle \, e_j
  + \sum_{\alpha \in A} \langle x, \td{e}_\alpha \rangle \, \td{e}_\alpha
  + \sum_{j=1}^n \lambda_j \langle x, e_j\rangle \, e_j.
$$
Consequently, for each $x \in S$,

$$
  \|T x\|^2= 1 + \sum_{j=1}^{n}
  \big(\lambda_j^2 + 2 \, \lambda_j \big)
  |\langle x, e_j \rangle|^2.
$$
Therefore, if $P$ is the finite range projection on $M_n$, then
for any $x \in S$
$$
  \left\|T P x\right\|=\left\|T x\right\|,
$$
and as $T P$ has finite range and therefore satisfies $\clg{AN}^*$,
then $T$ satisfies $\clg{AN}^*$. 
\end{proof}

\bigskip
\begin{lemma}\label{LPQRN} Let $H$ be
a separable Hilbert space. If $P,Q \in \clg{L}(H)$ are two
orthogonal projections such that, the dimension of their ranks and
null spaces are infinite, then $P$ and $Q$ are unitarily equivalent.
\end{lemma}

\begin{proof}
Since the rank and the null space of a projection are subspaces,
there exist unitary operators $U_1: P(H) \to Q(H)$ and $U_2:
{\rm Ker} \,P  \to {\rm Ker} \, Q$. Now, we define $U: H \to H $, such
that
$$
  U|_{P(H)}=U_1 \quad \text{and} \quad U|_{{\rm Ker} \, P}= U_2.
$$
Hence it is clear that $U$ as defined above is a unitary operator.
Moreover, if $x \in H$, then $x=x_1+x_2$ where $x_1 \in P(H)$ and
$x_2 \in {\rm Ker} \, P$. From the definition of $U_1$ and $U_2$, we have
$$
  Q U x=U x_1=U P x.
$$
Therefore, $P$ and $Q$ are unitarily equivalents.
\end{proof}

\medskip
\begin{theorem}
\label{PONR} Let $ Q \in \clg{L}(H)$ be an orthogonal projection.
Then, $Q$ satisfies the $\clg{AN}^*$ property
 if, and only if, the dimension of the null space or the dimension of the
rank of $Q$ is finite.
\end{theorem}

\begin{proof}
 If $\dim Q(H)< \infty$, then it is clear that $Q$ satisfies
$\clg{AN^*}$. Now, assume $\dim {\rm Ker} \,Q < \infty$. Then, we have
$Q= I - P$, where $P$ is a projection with finite rank. Therefore, by Lemma \ref{IFR}
$Q$ satisfies the $\clg{AN^*}$ property.

\medskip
Now, let us show that, if $Q$ satisfies the $\clg{AN}^*$ property, then
the dimension of the null space or the dimension of the
rank of $Q$ is finite, we show the contrapositive. 
Let $\dim Q(H)$ and $\dim {\rm Ker} \,Q$ be infinite.  We
consider two cases:

i)  H separable. In this case, by Lemma \ref{LPQRN}, we have that
$Q$ is unitarily equivalent to the orthogonal projection of Example
\ref{ENAN}, which does no satisfy the $\clg{AN^*}$ condition.
Consequently, by Theorem \ref{UEO} $Q$ does not satisfy $\clg{AN^*}$
either.

ii) H is not separable. In this point we will use the following
\medskip
\newline
\textit {It is not difficult to prove that: If $P \in \clg{L}(H)$ is an orthogonal projection
and $J$ is a subspace of $H$, such that $P(H) \subset J$, then $P|_J \in \clg{L}(J)$ is also 
an orthogonal projection and moreover:}
$$
   P(J)= P(H) \quad \text{and} \quad {\rm Ker} P|_J= J \cap {\rm Ker} P.
$$
If $Q(H)$ is countable, we take $J \subset H$ be a separable Hilbert space, such that $Q(H) \subset J$
and $\dim(Q(H)^\perp \cap J)= \infty$. Thus by the claim above, we have that $Q|_J$ 
is an orthogonal projection, $Q|_J \in \clg{L}(J)$ satisfying 
$$
   Q|_J(J)= Q(H) \quad \text{and} \quad {\rm Ker} \,Q|_J= J \cap {\rm Ker} \, Q.
$$
By the separable case (i), it follows that $Q|_J$ does not satisfy the $\clg{AN}^*$ property. 
Consequently, $Q$ does not satisfy $\clg{AN}^*$ either. 

\medskip
Finally, if $Q(H)$ is not countable, let $H_1 \subset Q(H)$ be an infinite countable subspace, 
and $Q_1: H \to H$ be an orthogonal projection on $H_1$.  Furthermore, let $N_1$ be an infinite countable 
subset of $Q(H)^\perp$ $(= {\rm Ker} \, Q)$, and consider
\begin{equation}
\label{DECOMP}
   H_2= H_1 \oplus N_1,
\end{equation}
which is a separable Hilbert space. As $H_1= Q_1(H) \subset H_2$, by the claim above it follows that 
$Q_1|_{H_2} \in \clg{L}(H_2)$ is an orthogonal on $H_2$ satisfying 
\begin{equation}
\label{DECOMP100}
  Q_1|_{H_2}(H_2)= Q_1(H)= H_1 \quad \text{and} \quad {\rm Ker} \,Q_1|_{H_2}= H_2  \cap {\rm Ker} \, Q_1.
\end{equation}
Since $H_1= Q_1(H) \subset Q(H)$ then ${\rm Ker} \, Q= Q(H)^\perp \subset Q_1(H)^\perp= {\rm Ker}\, Q_1$, we concluded 
by \eqref{DECOMP} and \eqref{DECOMP100} that $N_1 \subset {\rm Ker} \, Q_1|_{H_2}$. Conforming with the separable case 
(i), it follows that $Q_1|_{H_2} \in \clg{L}(H_2)$ is an orthogonal projection on $H_2$, which does not satisfy 
the $\clg{AN}^*$ property. Consequently, since for all $x \in H_2$, $x= x_1 + x_2$, with 
$x_1 \in H_1= Q_1(H) \supseteq Q(H)$, $x_2 \in N_1 \subset {\rm Ker} \, Q \subseteq {\rm Ker} \, Q_1$,  thus
$$
   \|Q x\|= \|Q_1 x \|= \|x_1 \|,
$$
neither $Q$ satisfies the $\clg{AN}^*$ property, and the proof is complete.
\end{proof}

\bigskip
Another important characterization of $\clg{AN^*}$ operators is given below, 
but one observes first that, if $P \in \clg{L}(H)$ is a positive operator,
then the inequality \eqref{positive}, with $y=Px$ gives
\begin{align*}
\|Px\|^{4} \le& \langle Px,x\rangle \langle P^{2}x,P x\rangle \le
\langle Px,x\rangle \| P^{2}x\|\,\|P x\| \le \langle Px,x\rangle
\, \| P\|\,\|P x\|^{2}.
\end{align*}
Therefore, it follows that
\begin{equation}\label{positive1}
\|Px\|^{2} \le \langle Px,x\rangle \,\| P\|.
\end{equation}
Now, we have the following 

\begin{lemma}
\label{LKP} Let $K \in \clg{L}(H)$ be a positive compact operator
and $\eta$ a positive real number, such that $\eta > \|K\|/2$.
Then, the operator
$$
  W:= \eta \, I - K,
$$
satisfies the $\clg{AN^*}$ property.
\end{lemma}
\begin{proof}
For any $x \in S$, we have
$$
  \|W x\|^{2}= \eta^{2} - \langle \big(2 \, \eta \, K - K^{2}\big)
  x,x\rangle.
$$
The condition $2 \, \eta > \|K\|$ and the inequality
\eqref{positive1} imply that
$$
  T:= 2 \, \eta \, K - K^{2}
$$
is a positive compact operator, in fact for each $x \in H$
$$
    \langle T x,x \rangle = 2 \eta  \langle K x,x \rangle - \|Kx\|^2 \geq 0,  
$$
where we have used that 
$$
    \|Kx\|^2 \leq  \langle K x,x \rangle  \|K\| \leq  \langle K x,x \rangle 2 \eta.
$$
Now, let $P_T$ be the positive square root of $T$, thus
$P_T$ is also a compact positive operator and
$$
\|W x\|^{2}= \eta^{2}-\|P_T x\|^{2}.
$$
Consequently, if $M$ is a closed subspace of $H$, then there
exists $x_0 \in S \cap M$, such that
$$
  [W|_M]=\sqrt{\eta^{2}-\|{P_T}|_M\|^{2}}=\sqrt{\eta^{2}-\|P_T x_0\|^{2}}=\|W
  x_0\|.
$$
\end{proof}

\bigskip
\begin{proposition}
Let $P \in \clg{L}(H)$ be an $\AN$ orthogonal projection
and $\eta> 1/2$. Then, the operator
$$
 T:= \eta \, I - P
$$
satisfies the property $\clg{AN^*}$.
\end{proposition}

\begin{proof}
Since $P^{2}=P$, we have for any $x \in S$
$$
\|Tx\|^{2}= \eta^{2} + (1 - 2 \, \eta)\|Px\|^{2}.
$$
Conjointly, as $P$ satisfies $\clg{AN}$ and $(1-2 \, \eta)< 0$, then
for any (closed) subspace $M$ of $H$, there exist $x_0 \in S \cap
M$, such that
$$
  \begin{aligned}
  {[T|_M]}^{2}=& \,\eta^{2}+(1-2\, \eta) \, \|P|_M\|^{2}
  \\
  =& \, \eta^{2}+(1-2 \, \eta) \, \|Px_0\|^{2}= \|Tx_0\|^{2}.
  \end{aligned}
$$
\end{proof}
\begin{proposition}
Let $P_1,P_2 \in \clg{L}(H)$ be $\clg{AN}^*$ orthogonal projections. Then
$P_1 \pm P_2$, $P_1 P_2$ and $P_2 P_1$ satisfy the $\clg{AN}^*$
property.
\end{proposition}

\begin{proof}
In fact, the proof follows with the following remark. If $P$ is an
orthogonal projection, which satisfies the $\clg{AN}^*$ property, then $P$ or
$I-P$ has finite rank. Therefore, if $P$ satisfies $\clg{AN}^*$ or
$P$ has finite rank, or we could write $P=I-K$, where $K$ is a
projection with finite rank. Then, we conclude the 
proof using Remark \ref{remark1} and Lemma \ref{IFR}.
\end{proof}

\section*{Acknowledgements}

The authors were partially supported by Pronex-FAPERJ through the grant
E-26/ 110.560/2010 entitled {\sl ``Nonlinear Partial Differential Equations''}.

\newcommand{\auth}{\textsc}


\end{document}